\documentclass[12pt]{article}
\author{Asl{\i} Deniz} 
\title{Convergence of Rays with Rational Argument in Hyperbolic Components}
\usepackage[english]{babel}
\usepackage{graphicx}
\usepackage{amscd}
\usepackage{amsfonts}
\usepackage{amscd}
\usepackage{latexsym}
\usepackage{epsfig}
\usepackage{amssymb,amsmath}
\usepackage[usenames]{color}
\usepackage{amscd}
\usepackage{psfrag}
\usepackage{amsthm}
\usepackage[linktocpage=true]{hyperref}
\theoremstyle{plain}

\newtheorem{thm}{Theorem}
\numberwithin{thm}{section}

\newtheorem{prop}[thm]{Proposition}
\newtheorem{lem}[thm]{Lemma}

\newtheorem*{thmmain}{Main Theorem}
\theoremstyle{definition}
\renewcommand\Im{\operatorname{Im}}
\renewcommand\Re{\operatorname{Re}}

\theoremstyle{definition}
\newtheorem{defn}[thm]{Definition}
\newtheorem{rem}[thm]{Remark}
\newtheorem{cor}[thm]{Corollary}
\date{}
\begin{document}
\maketitle
\begin{abstract}
In this paper, we use the Carath\'eodory Convergence Theory to prove a landing theorem of rays in hyperbolic components with rational arguments. Although the proof is done in the setting of a family of entire transcendental maps with two singular values, the method can be generalized to many other one parameter families, with some modifications.
\end{abstract}

\section{Introduction and setup}
We consider the dynamical system obtained by the iterates of an entire function $f:\mathbb{C}\rightarrow\mathbb{C}$. 
We denote $n$th iterate by $f^n=f \circ \overset{n}{\cdots} \circ f$. We are interested in understanding the long term behavior of orbits of points -sequences generated by the iterates of $f$- in terms of their initial conditions.
If the orbit of a point $z_0\in\mathbb{C}$ is finite, we say that $z_0$ is pre-periodic. If there exists $q\in\mathbb{N}$ satisfying $f^q(z_0)=z_0$, then $z_0$ is periodic; in particular, $z_0$ is a fixed point if $q=1$. Periodic points are classified in terms of their multipliers: $(f^q)'(z_0)$, and for most multipliers, local dynamics near periodic points can be reduced to a simple form. For example,
if $f(z_0)=z_0$ and $f'(z_0)=0$, $z_0$ is a superattracting fixed point. Then $f$ is conformally conjugate to $z\mapsto z^d$ for some $d\geq 2$. The conjugating map is called the B\"{o}ttcher coordinate, and is unique up to multiplication by the $(d-1)$st root of unity.
\\
\\
For  $f:\mathbb{C}\rightarrow\mathbb{C}$, the dynamical plane is partitioned into two totally invariant sets with respect to the behavior of the points; the set of points with stable behavior, i.e., the domain of normality of the iterates $\{f^n\}_{n}$, and its complement, i.e., the set of points with nonstable behavior. These two complementary sets are called the Fatou set, which we denote by $\mathcal{F}(f)$, and the Julia set, which we denote by $\mathcal{J}(f)$, respectively. A singular value of  $f$ is a point with the property that, in its neighborhood, at least one of the inverse branches of $f$ is not well-defined. There are two types of singular values: a critical value, and an asymptotic value. A critical value is the image of a branch point of $f$. If there exists a curve $\gamma:[0,\infty)\rightarrow\mathbb{C}$, $\lim_{t\rightarrow\infty}|\gamma(t)|=\infty$, such that $\lim_{t\rightarrow\infty}f(\gamma(t))=a$, then $a$ is called an asymptotic value. The behavior of singular values plays an essential role to understand the dynamics. For example, the number of singular values is associated with the number of stable domains.
\\
\\
The study of iteration began to receive more attention in the $80$'s, with the wide analysis of the quadratic family $Q_c(z)=z^2+c$, $c\in\mathbb{C}$, by Douady and Hubbard (\cite{douhub1982}). One of the main goals in \cite{douhub1982} is to understand the topology of the boundary of the Mandelbrot set $\mathcal{M}$ -the set of parameter values, for which the orbit of the singular value $z=c$ is bounded. In order to approach the boundary of $\mathcal{M}$, unbounded curves which come from infinity and accumulate on the boundary of $\mathcal{M}$ are constructed. These curves are called parameter rays. Construction of parameter rays requires analysis in dynamical plane:
For a quadratic polynomial, since the point at infinity is a superattracting fixed point with local degree $2$, in a neighborhood $\Omega_c$ of  $\infty$, the B\"ottcher coordinate conjugates the dynamics of $Q_c$  in $\Omega_c$ to $z\mapsto z^2$ in some neighborhood of $\infty$. A dynamic ray is defined as an inverse image of a radial line, which is equal to $\phi_c^{-1}(re^{2\pi i\theta}, r>R)$, for some $R\geq 1$. The complement of $\mathcal{M}$ is parametrized by the conformal isomorphism $\Phi:\mathbb{C}\backslash \mathcal{M}\rightarrow \mathbb{C}\backslash\overline{\mathbb{D}}$, $c\mapsto\phi_c(c)$. A parameter ray of argument $\theta$ is then given by $R_{\mathcal{M}}(\theta):=\Phi^{-1}(re^{2\pi i\theta},r>1)$. By construction, dynamic and  parameter rays have a relationship, namely, whenever a parameter is taken on a parameter ray of argument $\theta$, then the critical value $z=c$ is on the dynamic ray with the same argument. We say that $R_{\mathcal{M}}(\theta)$ lands if $\overline{R_{\mathcal{M}}(\theta)}\backslash R_{\mathcal{M}}(\theta)$ is a single point.
One of the main results Douady and Hubbard obtained is that the parameter rays with rational argument land on the boundary of $\mathcal{M}$ \cite[Chpt $8$]{douhub1982}. Since then, many alternative proofs were established to explain this phenomena (\cite{hubsch1994, mil1997, pet1998, petryd2000, sch1997}). In this article we give a new one, following the ideas of \cite{pet1998}, using  the Carath\'eodory Convergence Theory (see Section \ref{caratopology}), which relies on the convergence of marked domains. With some modifications, this proof can be applied to a more general setting.\\
\\
We illustrate the proof by a one parameter family of transcendental entire maps with two singular values, one of which is a critical value that is fixed at the critical point, and the other is a free asymptotic value that has only one finite preimage. Up to conjugacy  with a M\"{o}bius transformation, every such map can be written in the form
\begin{equation*}
f_a(z)=a(e^z(z-1)+1)
\end{equation*}
for some $a\in\mathbb{C}^*$ (see Theorem \ref{constructionoff_a}). In this parametrization, the critical point is fixed at $z=0$, while the asymptotic value is at $z=a$, and its finite preimage is at $z=1$. As the critical point is fixed for every $a\in\mathbb{C}^*$, it is a superattracting fixed point, and it has its own basin of attraction, which we denote by $A_a$.  Its immediate basin, denoted by $A_a^0$, is the connected component of $A_a$ which contains $0$. \\
\\
The asymptotic value can belong to $A_a$. This means, compared to the quadratic or the exponential family, there is a different class of hyperbolic components  consisting of parameter values for which the free singular  value is in $A_a$. (in Figure \ref{parameterplane_exp}, these components are shown in blue). We call \textit{the main hyperbolic component}, the special hyperbolic component in this class, given by
\begin{equation*}
\mathcal{C}^0:=\{a,\;a\in A_a^0\}.
\end{equation*}
 \begin{center}
\begin{figure}[htb!]
\includegraphics[height=6cm,width=6cm]{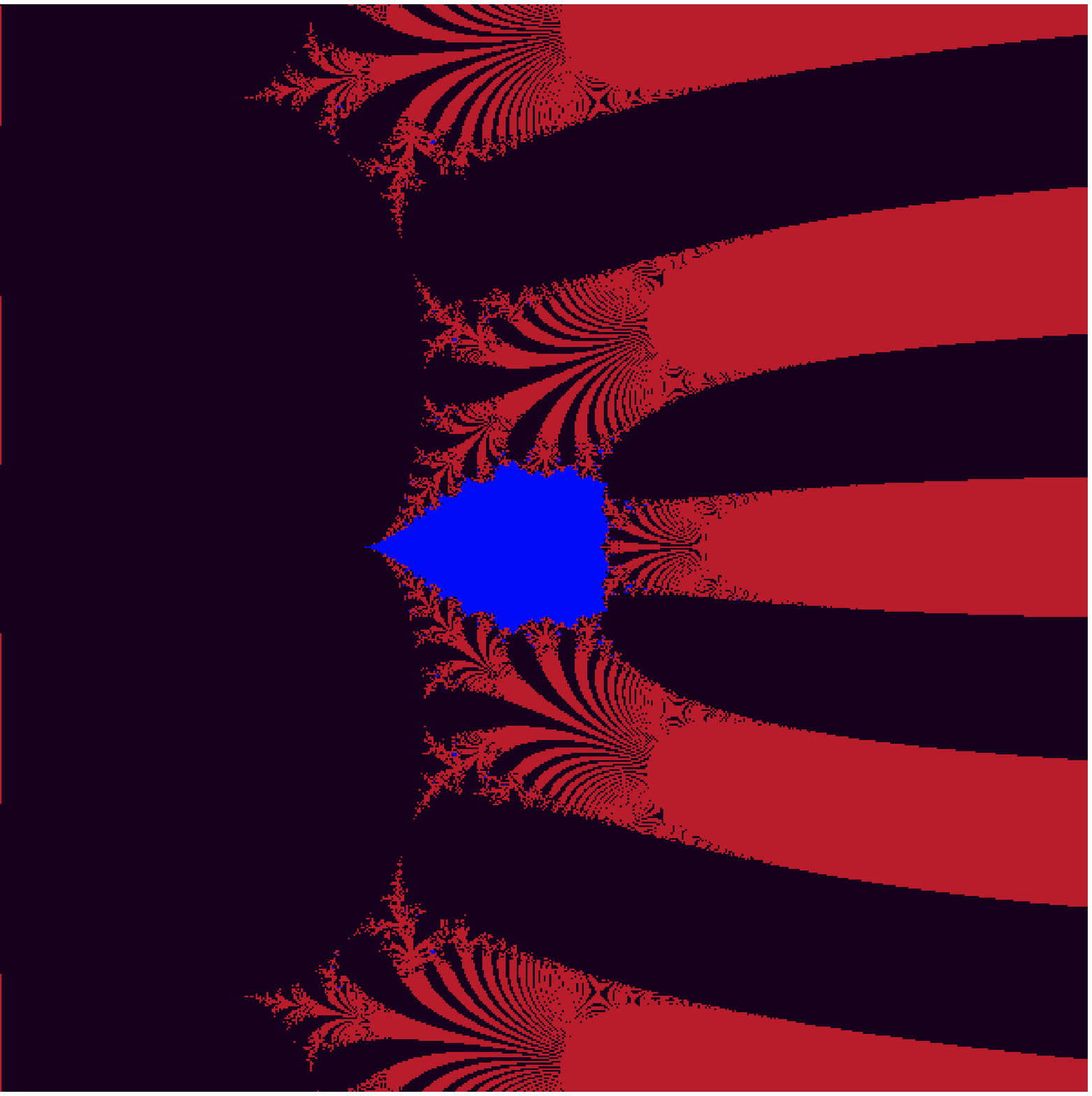}
\includegraphics[height=6cm,width=6cm]{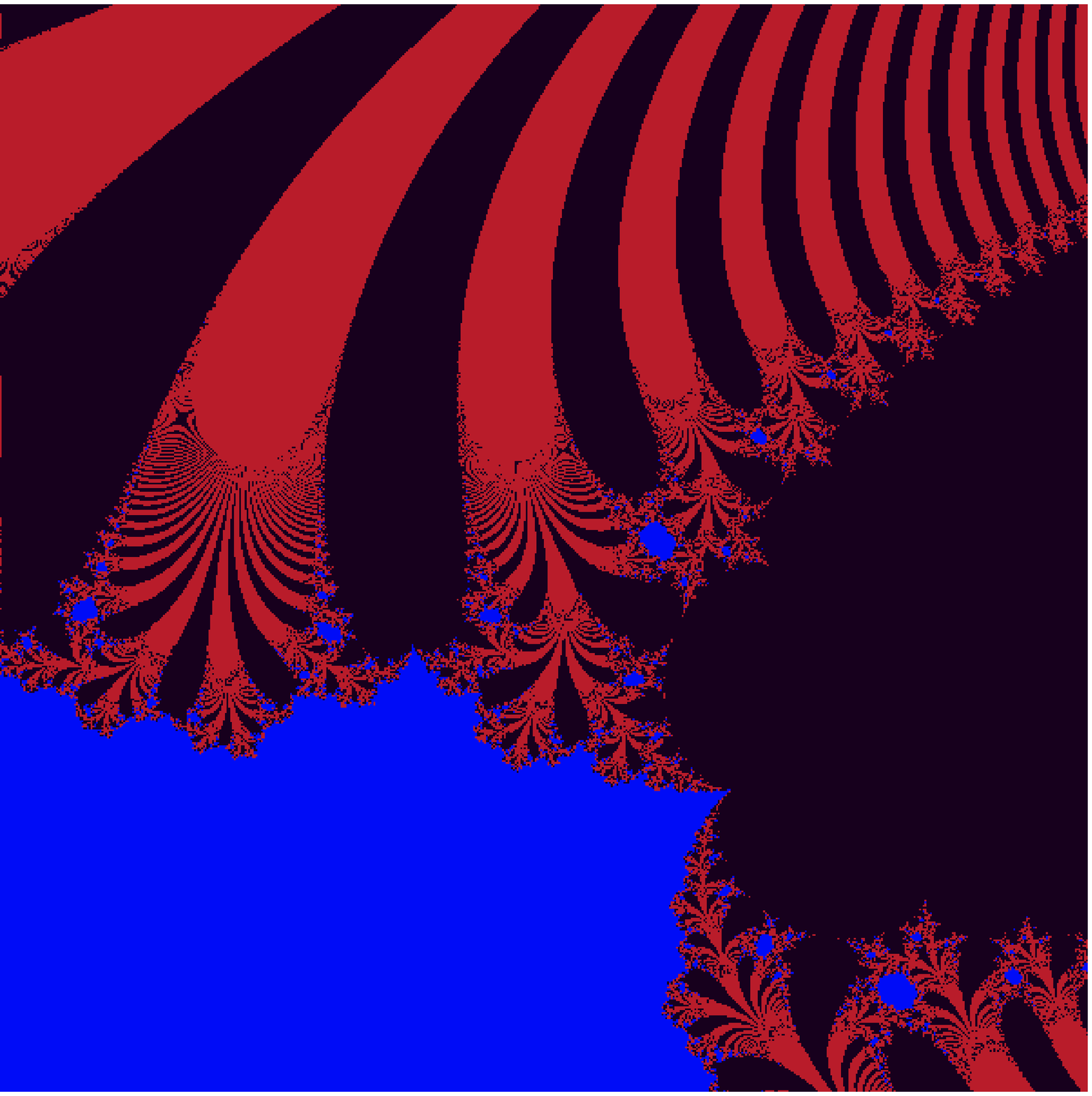}
\caption{\label{parameterplane_exp}\small{Parameter plane, left in $[10,10]\times[10,10]$-the blue component in the center is the main hyperbolic component, right in $[-0.5,2]\times[0.5,3]$. }}\label{parameterplane1}
\end{figure}
\end{center}
\noindent For each $a\in \mathcal{C}^0$, there exists a maximal domain $\Omega_a\in A_a^0$ and $\mathbb{D}_{r(a)}:=\{z;\;|z|<r(a)\}$ for which the restriction of the unique B\"{o}ttcher map $\phi_a|_{\Omega_a}:\Omega_a\rightarrow\mathbb{D}_{r(a)}$ -associated to the superattracting fixed point $0$-  is biholomorphic. Observe that $\Omega_a$ is unbounded while $f_a(\Omega_a)$ is bounded, and the boundary $\partial f_a(\Omega_a)$ contains $a$. Since $0$ is a simple critical point, $\phi_a$ conjugates the dynamics to $z\mapsto z^2$ in some neighborhood of $0$. The Green's function $g_a$, associated to $0$ is the subharmonic function on $\mathbb{C}$, which is equal to $\log|\phi_a(z)|$ on $\Omega_a$, and which extends to the whole plane by
\begin{equation*}
g_a(z) =
\left\{
	\begin{array}{ll}
		2^ng_a(z)=g_a(f_a^n(z))  & \mbox{if } z \in A_a^0,\\
		0 & \mbox{if } z\in\mathbb{C}\backslash A_a^0.
	\end{array}
\right.
\end{equation*}

\noindent By definition, $g_a$ depends continuously on the parameter. In $A_a^0$, we define the internal dynamic ray $R_{A_a^0}(\theta)$ of argument $\theta\in\mathbb{R}\slash\mathbb{Z}$, as the gradient line for $g_a$, which is equal to $\phi_a^{-1}(re^{2\pi i\theta},0<r<r(a))$ in $\Omega_a$. 
Note that $\phi_a$ has a unique continuous extension to the boundary of $\Omega_a\subset A_a^0\cup\{0\}$. For this extension we may define  $\phi_a(\infty):=-\phi_a(1)$, since both $\infty$ and $1$ are "mapped" to $a$ by $f_a$. We may then parametrize the main hyperbolic component by the biholomorphic map
\begin{eqnarray*}
\Phi:{\mathcal{C}^0}\cup\{0\}&\rightarrow&\mathbb{D}\\
a&\mapsto&\phi_a(\infty).
\end{eqnarray*}
The details are given in Section \ref{resultsf_a} (see Theorem \ref{parametrizationofthemainhyperboliccomponent}). According to this parametrization, the internal parameter ray $R_{{\mathcal{C}^0}}(\theta)$ of argument $\theta\in\mathbb{R}\slash\mathbb{Z}$ in $\mathcal{C}^0$ is given by the inverse image $\Phi^{-1}(re^{2\pi i\theta},0<r<1)$.

\begin{figure}[htb!]
\begin{center}
\def\svgwidth{12 cm}
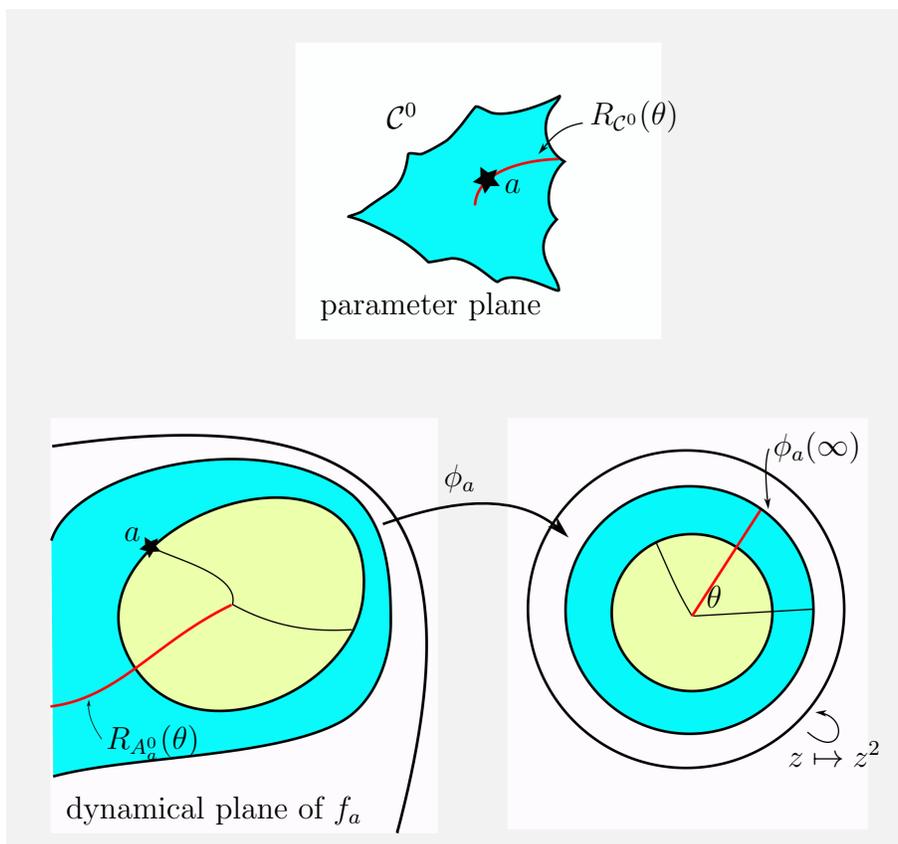
\caption{\small{Illustration of $\mathcal{C}^0$, dynamic and parameter rays.} }\label{parameterplane1}
\end{center}
\end{figure}

\noindent For $a\in R_{\mathcal{C}^0}(\theta)$, the corresponding internal dynamic ray $R_{A_a^0}(\theta)$ diverges to $\infty$. By construction, $a\in R_{{\mathcal{C}^0}}(\theta)\cup R_{{\mathcal{C}^0}}(\theta+\frac{1}{2})$ if and only if $a\in R_{A^0_a}(2\theta)$.  The correspondence between dynamic and parameter rays is illustrated by Figure \ref{parameterplane1}. As usual, we say that a ray is rational, if its argument is rational, i.e., $\theta$ satisfies $2^{l+q}\theta\equiv 2^{l}\theta \pmod1$ for some integers $l\geq 0$, $q\geq 1$. We say that $R_{\mathcal{C}^0}(\theta)$ (respectively $R_{A_a^0}(\theta)$) lands if $\overline{R_{\mathcal{C}^0}(\theta)}\backslash R_{\mathcal{C}^0}(\theta)$ (respectively $\overline{R_{A_a^0}(\theta)}\backslash R_{A_a^0}(\theta)$) is a single point. We are interested in the case when $\theta$ is rational, and our main result is the following.

\begin{thmmain}\label{thm} Every rational internal parameter ray lands at a boundary point of the main hyperbolic component $\mathcal{C}^0$. The landing point is either a parabolic, repelling, or a Misiurewicz parameter. More precisely, for $\theta\in\mathbb{Q}\slash\mathbb{Z}$, such that $2^{l+q}\theta\equiv 2^{l}\theta \pmod1$ for some minimal integers $l\geq 0$, $q\geq 1$, and for the landing point  $a\in\partial {\mathcal{C}^0}$ of  $R_{\mathcal{C}^0}(\theta)$,
\begin{itemize}
\item[1)] if $l=0$, then $f_a$ has a $q$-periodic parabolic basin containing $a$,
\item[2)] if $l=1$, then the asymptotic value $a$ belongs to a repelling periodic cycle of period $q$ on $\partial A_a^0$ with $f_a^{q-1}(a)=1$, and
\item[3)] if $l>1$, then the asymptotic value $a$ is (strictly)  preperiodic with $f_a^{l-1}(a)=f_a^{l-1+q}(a)\in\partial A_a^0$.
\end{itemize}
Moreover, for each case, $q$ is the exact period.
\end{thmmain}

\noindent The situation $2)$ in the statement of our theorem is a special case which does not occur in the quadratic family for a landing parameter of a rational ray on the boundary of the Mandelbrot set. Because in such a case the cycle containing the critical value would be superattracting, and it is a contradiction. The situation $2)$ does not occur in the exponential family for any parameter value either. Because the asymptotic value for an exponential map can not be in a periodic cycle. This extra case is due to the existence of the finite preimage of the asymptotic value of $f_a$, which allows the asymptotic value to be in a periodic cycle. If a family has a free singular value which has finite preimage(s), the singular value can be in a cycle for a parameter on the boundary of hyperbolic components, hence a situation as given by $2)$ can occur. 
\\
\\
In fact, what we illustrate by the specific family $f_a$ is the phenomena that by varying the parameter, the free singular value(s) leaves a persistent superattracting immediate basin along rational rays. In this respect, although we present the proof by a specific family, taking the difference in parametrization of the main hyperbolic component and the type of the free singular value into account, our proof applies to many different families. For example consider the family of polynomials
\begin{equation}
f_{a,n}(z)=a\Big(\big(1+\frac{z}{n}\big)^n(z-1)+1\Big),\;\;\;\;n\geq 2,
\end{equation}
for fixed $n$. Here, $f_{a,n}$ can be considered as a polynomial analog of $f_a$, since for any $f_{a,n}$, there are two singular values, which are both critical, and one of which is fixed, and the other is free. 
Similar to $f_a$, we define the main hyperbolic component for $f_{a,n}$, as the set of parameters for which both critical values are in the same Fatou component. The statement of the Main Theorem is valid for $f_{a,n}$, by replacing "asymptotic value" by "critical value". Note also that there is a strong link between $f_{a,n}$ and $f_a$: the sequence $\{f_{a,n}\}_n$ converges uniformly to $f_a$ on compact subsets of $\mathbb{C}$. Reader can also consider $f_{a,n}$ for fixed $n$ in the proof, instead of $f_a$, with attention to the difference coming from the nature of two different type of singular values.
\\
\\
Now we turn back to our family $f_a$ in consideration and give the idea of the proof as follows: For $\theta\in\mathbb{Q}\slash\mathbb{Z}$, we take a sequence of parameters $\{a_n\}_n\in R_{\mathcal{C}^0}(\theta)$, which converges to a limit point $a\in\partial \mathcal{C}^0$. For each element $a_n$ of the sequence, we observe the dynamical plane, and study the convergence of Fatou domains containing the asymptotic value, taking this as a marked point. We see that in the limit, the asymptotic value $a$ satisfies a particular equation, whose solutions form a discrete set. However, as being an accumulation set of a connected set, $\overline{R_{\mathcal{C}^0}(\theta)}\backslash R_{\mathcal{C}^0}(\theta)$ is connected, hence it consists only of one point.
\\
\\
Let us give the structure of the article: In the second section, we introduce the family of transcendental entire maps in consideration. In the third section, we give the main definitions and results regarding 
Carath\'eodory Topology. The last section is devoted to the proof of the Main Theorem.

\subsubsection*{Acknowledgements} I would like to thank Carsten Petersen, for introducing me the problem and the tool,  and N\'uria Fagella for many useful comments. I would also like to thank IMPAN-Insititute of Mathematics, Polish Academy of Science, for hospitality and support during the preparation of this paper. This work was supported by Marie Curie RTN 035651-CODY and Roskilde University.

\section{Overview of the family $f_a$ with related results}\label{resultsf_a}
Having two singular values,  $f_a$ is in the class of entire transcendental maps with a finite number of singular values.  As proved in \cite{erlyu1992} adapted to our family,  all Fatou components are  preperiodic, simply connected, and there are at most two attracting basins or indifferent cycles. Those are the general properties for the members of the family in consideration, which we will use in the proof. 

\subsection{Characterization of $f_a$}
The following theorem shows that $f_a$ includes all entire transcendental maps which share some given dynamical properties.
\begin{thm}\label{constructionoff_a}
Any entire transcendental map of finite order whose singular values are 
\begin{itemize}
\item[i.] one asymptotic value with only one finite preimage, and
\item[ii.] one simple critical value which is fixed at the critical point (normalized at $0$)
\end{itemize}
is affine conjugate to one of the form
\begin{equation*}
f_a(z)=a(e^z(z-1)+1),\;\;a\in\mathbb{C}^*.
\end{equation*}
Moreover, $f_a$ contains a unique representative of each conformal conjugacy class.
\end{thm}
\noindent The proof uses the same idea as in \cite[Thm 3.1]{berfag2010}. A sketch is as follows.
\begin{proof}
Let $g(z)$ be a map as in the statement of the theorem. Suppose $a\in\mathbb{C}$ is the asymptotic value with finite preimage $u$. Since $g(z)-a$ has exactly one solution at $z=u$, then by the Hadamard Factorization Theorem, $g$ is of the form:
\begin{equation*}
g(z)=(z-u)^me^{h(z)}+a,
\end{equation*}
where $h(z)$ is a polynomial and $m\in\mathbb{N}$. The map $g$ has one simple critical point at $0$, so the equality
\begin{equation*}
g'(z)=e^{h(z)}(z-u)^{m-1}\big{(}m+h'(z)(z-u)\big{)}=0
\end{equation*}
gives $m=1$, and $h'(z)=c$, and hence $h(z)=cz+d$. Therefore, we obtain
\begin{equation*}
g(z)=(z-u)e^{cz+d}+a.
\end{equation*}
As $z=0$ is a superattracting fixed point, $e^d=\frac{a}{u}$, and $c=\frac{1}{u}$. This yields:
\begin{equation*}
g(z)=(z-u)e^{\frac{z}{u}}\frac{a}{u}+a=a((\frac{z}{u}-1)e^{\frac{z}{u}}+1).
\end{equation*}
Redefining the variable: $w=\frac{z}{u}$, we obtain $f_a(w)=a(e^w(w-1)+1)$, requiring $a\neq 0$. Observe that this construction forces the finite preimage of the asymptotic value to be $1$. Finally, since any conformal conjugacy between two members of $f_a$ must fix $0$, $1$, and $\infty$, it follows that it is the identity map.
\end{proof}

\subsection{Main hyperbolic component}
In this section, we are going to present some results regarding the main hyperbolic component $\mathcal{C}^0$, and related dynamical properties. For example, $\mathcal{C}^0$ consists of parameters for which the Julia set is  \textit{a Cantor bouquet} (see Proposition \ref{totallyinvariantcomponent}). We recall that a Cantor bouquet is a special form of Julia set, which can be described as a collection of disjoint continuous curves in the dynamical plane tending to $\infty$ in a certain direction, with distinguished endpoints. The following two theorems are going to be used in the proof of Proposition \ref{totallyinvariantcomponent}.
\begin{thm}(Baranski-Jarque-Rempe, \cite[Thm 1.5]{barxarem2011})\label{thmbarxarem}
Let $f:\mathbb{C}\rightarrow\mathbb{C}$ be a transcendental entire function of finite order, which has one completely invariant Fatou component. Then the Julia set $\mathcal{J}(f)$ is a Cantor bouquet.
\end{thm}
\begin{thm}\label{Eremenko, Lyubich}(Eremenko-Lyubich, \cite[Lem 11]{erlyu1992}) Assume that a transcendental entire function has a totally invariant domain $D$. Then all singular values are contained in $D$.
\end{thm}
\begin{prop}\label{totallyinvariantcomponent} The following two statements are equivalent:

\begin{itemize}
\item[i.] $a\in \mathcal{C}^0$.
\item[ii.] $A_a$ consists of an unbounded, totally invariant component, and the Julia set is a Cantor bouquet.
\end{itemize}

\end{prop}
\begin{proof}
($i\Rightarrow ii$) Suppose $a\in \mathcal{C}^0$. First we will show $A_a=A_a^0$. We take a simple curve $\gamma\subset A_a^0$ joining $a$ and $0$. Let the curves $\widehat{\gamma}_0$ and $\widehat{\gamma}$ be the components of $f_a^{-1}(\gamma)$ such that $\widehat{\gamma}_0$ connects $\infty$ to $0$ and $\widehat{\gamma}$ connects $1$ to $0$. Since $\widehat{\gamma}\cap\gamma\neq\emptyset$ and $\widehat{\gamma}_{0}\cap\gamma\neq\emptyset$, then the curves $\widehat{\gamma}$, $\widehat{\gamma}_0$ and $\gamma$ are in the same component. The other preimages of $\gamma$, say $\widehat{\gamma}_i$, $i\in\mathbb{Z}\backslash\{0\}$ are curves connecting $\infty$ to the other preimages of $0$. 
In a neighborhood $\Omega$ of $\infty$, $\widehat{\gamma}_i\cap\Omega$ are almost $2\pi i$ translates of each other. This is because for large $|z|$ values $f_a(z+2\pi i)\approx f_a(z)$. Moreover the inverse image of a small neighborhood of $a$ will contain a left half plane, say $H_{-}$, and hence $\widehat{\gamma}_i\cap H_{-}\neq\emptyset$ for $i\in\mathbb{Z}$. This means $A_a^0$ is backward invariant, and thus $A_a^0$ is totally invariant. In other words, $A_a$ consists of only one component, i.e., $A_a=A_a^0$. Moreover, since $A_a$ contains the unbounded curves $\widehat{\gamma}_i$, it is unbounded.  A Cantor bouquet structure of the boundary $\partial A_a$ is guaranteed by Theorem \ref{thmbarxarem} (see Figure \ref{Acaptured}).\\
\\
($ii\Rightarrow i$) Now suppose that $A_a=A_a^0$. In this case Theorem \ref{Eremenko, Lyubich} guarantees that the asymptotic value $a$ is contained in $A_a^0$ as well, that is, $a\in\mathcal{C}^0$. 
\end{proof}
\begin{figure}
 \begin{center}
\includegraphics[height=6cm,width=6cm]{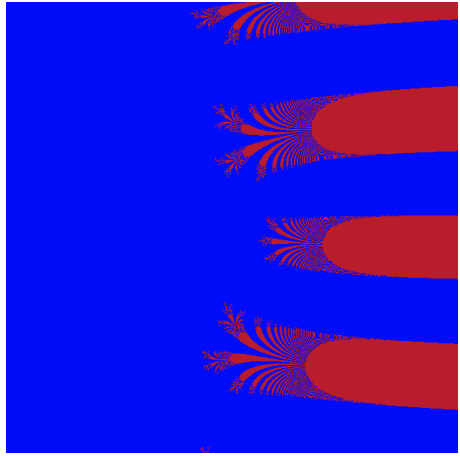}
\end{center}
\caption{\label{Acaptured}\small{Dynamical plane for the parameter $(0.4+0.6i)\in\mathcal{C}^0$ in $[-10,10]\times[-10,10]$.}}
\end{figure}

\noindent Set $\mathbb{D}_{r(a)}:=\{z;\;\;|z|<r(a)\}$. Recall that for each $a\in\mathcal{C}^0$, there exists $r(a)$ and a maximal domain $\Omega_a\in A_a^0$, for which the restriction of the unique B\"ottcher map $\phi_a|_{\Omega_a}:\Omega_a\rightarrow\mathbb{D}_{r(a)}$ -which conjugates $f_a$ to $z\mapsto z^2$-  is biholomorphic. 
 In a neighborhood of $0$, $\phi_a$  can be written as
\begin{eqnarray}\label{botsection}
\phi_a(z)=z\Big{(}\frac{f_a(z)}{z^2}\Big{)}^{1/2}\Big{(}\frac{f_a^2(z)}{(f_a(z))^2}\Big{)}^{1/2^2}...\Big{(}\frac{f_a^n(z)}{(f_a^{n-1}(z))^2}\Big{)}^{1/2^n}...
\end{eqnarray}
From the Taylor expansion, one can see that
\begin{equation*}
\phi_a(z)=\frac{a}{2}z+O(z^2).
\end{equation*}
This map has a continuous extension to the boundary $\partial \Omega_a\subset\widehat{\mathbb{C}}$. For this extension, we have $\phi_a(\infty)=-\phi_a(1)$, as explained in the introduction. With this identification, we parametrize the main hyperbolic component given by the following theorem.
\begin{thm}\label{parametrizationofthemainhyperboliccomponent} 
There exists a biholomorphic map from $\mathcal{C}^0\cup\{0\}$ to $\mathbb{D}$, which is given by:
\begin{eqnarray*}
\Phi:\mathcal{C}^0\cup\{0\}&\rightarrow&\mathbb{D}\notag\\
a&\mapsto &\phi_a(\infty).
\end{eqnarray*}
\end{thm}
\begin{proof}
 By using (\ref{botsection}), 
\begin{equation*}
\Phi(a)=\phi_a(\infty)=-\phi_a(1)=-a^{1/2}\Big{(}\frac{a}{2}+O(a^2)\Big{)}^{1/2^2}...\Big{(}\frac{a}{2}+O(a^2)\Big{)}^{1/2^n}...
\end{equation*}
As $a\rightarrow 0$ 
\begin{equation}\label{degreofphi}
\Phi(a)\approx -\frac{a}{\sqrt{2}}.
\end{equation}
Hence $0$ is a removable singularity, i.e., $\Phi$ extends to $0$ with $\Phi(0)=0$.\\
\\
Now we will show that $\Phi$ is proper. For any sequence $\{a_n\}_n\subset\mathcal{C}^0$ tending to a limit point $a\in\partial \mathcal{C}^0$, consider the sequence $\{g_{a_n}\}_n$ of Green's functions  associated to the superattracting fixed point $0$. Then we have 
\begin{equation*}
|\Phi(a_n)|=e^{g_{a_n}(1)}\rightarrow e^{g_a(1)}\;\;\;\mathrm{as}\;\;\;a_n\rightarrow a\in\mathcal{C}^0,
\end{equation*}
since the Green's function is continuous with respect to the parameter. Moreover, at any point $a\in\partial \mathcal{C}^0$, $g_a(1)=0$, since $a\notin A_a^0$. So $|\Phi(a_n)|\rightarrow 1$ as $a_n\rightarrow a\in\partial \mathcal{C}^0$, hence $\Phi$ is proper. Therefore, $\Phi$ is a covering map, and the degree is $1$ by (\ref{degreofphi}). In other words, the map $\Phi:\mathcal{C}^0\cup\{0\}\rightarrow\mathbb{D}$ is biholomorphic.
\end{proof}

\begin{cor}
This theorem guarantees that $\mathcal{C}^0$ is connected and $\mathcal{C}^0\cup\{0\}$ is simply connected.
\end{cor}
\begin{thm}\label{mainhyperboliccomponentisbounded} $\mathcal{C}^0$ is bounded.
\end{thm}
\begin{proof}
Set $\mathbb{S}_r:=\partial \mathbb{D}_r=\{z;\;|z|<r\}$, and take $h_a(z):=a\frac{z^2}{2}$. We will find $r>0$, such that for all $z\in\mathbb{S}_r$,
\begin{equation}\label{condition}
|h_a(z)-z|>|h_a(z)-f_a(z)|.
\end{equation}
This will mean that $A_a^0\subset\mathbb{D}_r$, and will give a relationship between $a$ and $r$; we will see that for large values of $|a|$, $r$ is very small, so $\mathbb{D}_r$ and hence $A_a^0$ can not contain $a$.
\begin{itemize}
\item[i.] For $r\geq\frac{2}{|a|}$ and $|z|=r$,  
\begin{equation*}
|h_a(z)-z|=|\frac{az^2}{2}-z|=r|\frac{az}{2}-1|\geq r(\frac{|a|r}{2}-1).
\end{equation*}
\item[ii.] 
\begin{equation*}
 \begin{split}
    |h_a(z)-f_a(z)|&=|\frac{az^2}{2}-\sum_{n=0}^{\infty} a z^{n+2} \frac{(n+1)}{(n+2)!}\\
    &\leq |a|\sum_{n=0}^{\infty}  r^{n+3} \frac{(n+2)}{(n+3)!}\\
    &\leq|a|  \sum_{n=0}^{\infty}  r^{n+3} \frac{1}{(n+2)!}\\
   &=|a| r \sum_{n=0}^{\infty}  r^{n+2} \frac{1}{(n+2)!}\\
   & = |a|r(e^r-1-r).
    \end{split}
\end{equation*}
\end{itemize}
From i. and ii., $r$ and $a$ values satisfying 
\begin{equation*}
r(\frac{|a|r}{2}-1)>|a|r(e^r-1-r),
\end{equation*}
or equivalently
\begin{equation*}
1<|a|(\frac{3r}{2}-e^r+1)
\end{equation*}
also satisfies (\ref{condition}). Provided that  $\frac{3}{2}r-e^r+1>0$,  ($r\in (0, 0.7626)$, approximately) we obtain
\begin{equation}\label{relationar}
|a|>\frac{1}{1+\frac{3r}{2}-e^r}.
\end{equation}
By numerical computation,
\begin{equation*}
\min_{r\in(0,0,7626)}\frac{1}{1+\frac{3r}{2}-e^r}\approx9.2324.
\end{equation*}
 is obtained at $r\approx0.4054$. This means  for $|a|> 9.2324$, $f_a$ maps the circle $\mathbb{S}_r$ with $r\approx 0.4054$ outside itself. Thus $A_a^0$ is contained in the disk $\mathbb{D}_r$ with $\approx0.4054$  for parameter values $|a|>9.2324$. Therefore, the asymptotic value is not in $A_a^0$ for $|a|>9.2324$. Thus $\mathcal{C}^0$ is bounded.
\end{proof}


\section{Carath\'eodory topology}\label{caratopology}
The Carath\'eodory topology, introduced in \cite{car1912}, deals with sequences of domains and their convergence properties. In this section, we state the main definitions and prove some auxilary results which will be used in the proof of the Main Theorem. We start by defining convergence of a sequence of marked domains, following \cite{mcmul1994}.
\begin{defn}\label{definitioncar}  Consider the set of pointed domains ${\mathcal{D}}=\{(W,w),\;w\in W\subset\mathbb{C}\}$. We say a sequence $\{(U_n,u_n)\}_n\subset\mathcal{D}$ converges to $(U,u)$ in the sense of Carath\'eodory, if and only if the following holds:

\begin{itemize}
\item[i)] $u_n\rightarrow u$ as $n\rightarrow \infty$,
\item[ii)] for any compact set $C\subset U$ and for all sufficiently large $n\in\mathbb{N}$,  $C\subset U_n$, and
\item[iii)] for any connected, open set $N$ containing $u$,  if $N\subset U_n$ for infinitely many $n$, then $N\subset U$,
\end{itemize}
and we write $(U_n,u_n)\stackrel{Car} \longrightarrow (U,u)$. 
\end{defn}

\noindent The theory  is a powerful tool to relate the analytic behavior of maps to the geometric properties of their ranges. This relation is given by the following theorem.

\begin{thm}\label{carthm}(Carath\'{e}odory \cite{car1912}) Let $\{(U_n,u_n)\}_n$ be a sequence of pointed disks in $\mathcal{D}$, and let $\psi_n:\mathbb{D}\rightarrow U_n$ be the sequence of biholomorphic maps normalized such that $\psi_n(0)=u_n$, $n\in\mathbb{N}$ and $\psi_n'(0)>0$. Then $\psi_n$ converges to a univalent map $\psi:\mathbb{D}\rightarrow U$, uniformly on compact subsets of $\mathbb{D}$, where $\psi(0)=u$ and $\psi'(0)>0$, if and only if $(U_n,u_n)\stackrel{Car} \longrightarrow (U,u)$.
\end{thm}

\noindent A proof can be found, for example, in \cite[Sect 122]{car1998}, or \cite[Sect 1.4]{pom1975}.

\begin{rem}\label{limitdomainsimplyconnected}  The presence of the univalent map $\psi:\mathbb{D}\rightarrow U$ in Theorem \ref{carthm}  guarantees simple connectedness of the limit domain $U$.
\end{rem}

\noindent Convergence of pointed domains in the sense of Carath\'{e}odory  is related to Hausdorff convergence properties of their complements. We first recall the definition of Hausdorff convergence.

\begin{defn}  Let $\rho$ be the spherical metric in $\widehat{\mathbb{C}}$, and let $A$ and $B$ be nonempty compact sets in $\widehat{\mathbb{C}}$. Let $\Omega_{\epsilon}(A)$ and $\Omega_{\epsilon}(B)$ denote $\epsilon$-neighborhoods of the sets $A$ and $B$, respectively. \textit{The Hausdorff $\rho$-distance} $d_H(A,B)$ is given by:
\begin{equation*}
d_H(A,B)=\inf\{\epsilon>0:A\subset \Omega_{\epsilon}(B), B\subset \Omega_{\epsilon}(A)\}.
\end{equation*}
The metric defined by this distance is called \textit{the Hausdorff $\rho$-metric} on the set of all nonempty compact subsets of $\widehat{\mathbb{C}}$. Note that $d_H(.,.)$  depends on $\rho$.
\end{defn}
\noindent For nonempty compact sets $K$ and $K_n$, $K_n\rightarrow K$ in the Hausdorff metric, if for arbitrary $\epsilon>0$, there exists $N\in\mathbb{N}$, such that for all $n\geq N$, $K\subset \Omega_{\epsilon}(K_n)$, and $K_n\subset \Omega_{\epsilon}(K)$. Equivalently, $d_H(K,K_n)\rightarrow 0$ as $n\rightarrow \infty$. The space of all nonempty compact subsets of the sphere is a compact metric space. For details, see for example \cite[Prop 2.1, Cor 2.2]{mcmul1996}.\\
\\
We next give the relation between Carath\'{e}odory convergence of pointed domains and Hausdorff convergence of their complements, which actually is an alternative to Definition \ref{definitioncar}.
\begin{thm}\label{carconalter}
Let $\{U_n\}_n$ be a sequence of domains. The sequence of pointed domains $\{(U_n,u_n)\}_n$ converges to $(U,u)$ as $n\rightarrow\infty$ in the sense of Carath\'{e}odory, if and only if $u_n\rightarrow u$, and for any subsequence $K_{n_k}=\widehat{\mathbb{C}}\backslash U_{n_k}$ convergent to some compact set $K$ in the Hausdorff topology on compact sets of the sphere, $U$ is the component of $\widehat{\mathbb{C}}\backslash K$, which contains $u$. 
 \end{thm}
\begin{proof}
Suppose $(U_n,u_n)\stackrel{Car} \longrightarrow (U,u)$. Set $K_n=\widehat{\mathbb{C}}\backslash U_n$. Since the space of compact sets of $\widehat{\mathbb{C}}$ equipped with the Hausdorff $\rho$-metric is compact, there exists a convergent subsequence $\{K_{n_k}\}_k$ and a compact set $K$, such that $K_{n_k}\rightarrow K$ as $k\rightarrow\infty$. We are going to show that in this case, $U$ is the connected component of $\widehat{\mathbb{C}}\backslash K$, which contains $u$. 
\\
\\
\noindent Let $U'$ be the connected component of $\widehat{\mathbb{C}}\backslash K$, which contains $u$. Recall that by $d_{\rho}(.,.)$ we denote the spherical distance, and by $d_H(.,.)$ the Hausdorff $\rho$-distance defined on the compact sets of $\widehat{\mathbb{C}}$. We want to prove that $U'=U$:
\begin{itemize}
\item[i.] Take an arbitrary point $z\in U$. Let $C\subset U$ be a compact and connected set containing $\{z,u\}$ in $U$. Then there exists $\epsilon>0$ such that $d_{\rho}(C, K)=2\epsilon>0$. By convergence, $C\subset U_n$, for sufficiently large $n\in\mathbb{N}$. On the other hand, increasing $n$ if necessary,  $d_{H}(K_{n_k},K)<\epsilon$. This means for sufficiently large $n_k$, $d_{\rho}(C,K_{n_k})>\epsilon$, i.e.,  $C\cap K_{n_k}=\emptyset$. This implies $C\cap K=\emptyset$, that is, $C\subset\widehat{\mathbb{C}}\backslash K$. Since $C$ is connected, then it is contained in a connected  component of $\widehat{\mathbb{C}}\backslash K$. Since $u\in C$, then $C\subset U'$. Since $z$ is arbitrary in $U$, we obtain  $U\subset U'$.
\item[ii.] Take an arbitrary point $z\in U'$. Let $N$ be an open and connected set containing $\{z,u\}$, which is relatively compact in $U'$. Then there exists $\epsilon>0$ such that $d_{\rho}(\overline{N}, K)>2\epsilon$. Increasing $k$ if necessary, $d_H(K,K_{n_k})<\epsilon$. Hence $d_{\rho}(\overline{N},K_{n_k})>\epsilon$. This means $N\subset U_{n_k}$. Thus $N\subset U_n$ for infinitely many $n$, and hence $N\subset U$. Since $z$ is arbitrary in $U'$, $U'\subset U$.
\end{itemize}

\noindent By $i.$ and $ii.$, we obtain $U=U'$.\\
\\
Now we prove the converse of the statement: Suppose $u_n\rightarrow u$ as $n\rightarrow \infty$, and there exists a set $U$, such that $u\in U$, and $U$ is the connected component of $\widehat{\mathbb{C}}\backslash K$, where $K$ is the Hausdorff limit of any subsequence $K_{n_k}:=\widehat{\mathbb{C}}\backslash U_{n_k}$. We will show that  $(U_n,u_n)\stackrel{Car} \longrightarrow (U,u)$. 

\begin{itemize}
\item[i.] Let $C$ be a compact set in $U$. We claim that there exists $N\in\mathbb{N}$ such that  for $n\geq N$, $C\subset U_{n}$, in other words, there exists $\epsilon>0$ such that $d_{\rho}(C,K_n)>\epsilon$. If not, there exists a subsequence $K_{n_k}$ and $z_k\in K_{n_k}$ such that $d_{\rho}(z_{k},C)\leq \epsilon$, where $z_k\rightarrow z\in K$ as $k\rightarrow\infty$. Passing to  subsequence if necessary, suppose $K_{n_k}\rightarrow K$ as $k\rightarrow \infty$ and $z_k\rightarrow z\in K$ with $d_{\rho}(z,C)\leq \epsilon$. This contradicts $K\cap U=\emptyset$.
\item[ii.] For an open and connected set $N$ containing $u$, such that $N\subset U_n$ for infinitely many $n$,  we will show that $N\subset U$. Let $\{U_{n_k}\}_k$ be a subsequence such that  $N\subset U_{n_k}$. Then $N\cap K_{n_k}=\emptyset$, where $K_{n_k}=\widehat{C}\backslash U_{n_k}$. Taking a subsequence if necessary, $K_{n_k}\rightarrow K$, so $N\cap K=\emptyset$. And since $u\in N$, $N\subset U$.
\end{itemize}
By $i.$ and $ii.$, we conclude that $(U_n,u_n)\stackrel{Car} \longrightarrow (U,u)$. 

\end{proof}

\noindent We focus on convergence to a proper subdomain of the plane. In order to do so, we restrict our attention to \textit{precompact} sequences in 
\begin{equation}\label{D^*}
{\mathcal{D}^*}:=\{(W,w);\;\;w\in W\subsetneq\mathbb{C}\}\subset\mathcal{D}.
\end{equation}
\begin{defn}
A sequence of pointed domains $\{(U_n,u_n)\}_n$ is called precompact in $\mathcal{D}^*$, if and only if  every subsequence has a convergent subsubsequence  in $\mathcal{D}^*$, in the sense of Carath\'eodory. 
\end{defn}
\noindent Precompactness is characterized as follows.
\begin{prop}\label{precompactness1}
The sequence of pointed domains $\{(U_n,u_n)\}_n$ is precompact in $\mathcal{D}^*$, if and only if  $\{u_n\}_n$ is bounded,  and  there exists $1<L<\infty$, such that
\begin{equation}\label{precompactnessbound1}
\frac{1}{L}\leq d_E(u_n,\partial U_n)\leq L,
\end{equation}
 where $d_E$ denotes the Euclidian distance.
\end{prop}
\begin{proof}
First observe that any given subsequence $\{(U_{n_k},u_{n_k})\}_k$ is divergent, if $u_{n}\rightarrow\infty$. So boundedness of $\{u_n\}_n$ is necessary for precompactness. Now suppose there is no $L\in(1,\infty)$, such that $\frac{1}{L}\leq d_E(u_n,\partial U_n)\leq L$. Then obviously, for any subsequence $\{U_{n_k}\}_k$ and $K_{n_k}:=\widehat{\mathbb{C}}\backslash U_{n_k}$, 
\begin{equation}\label{contradictioninequality}
0\leq d_E(u_{n_k},K_{n_k})\leq\infty.
\end{equation}
 
\noindent Since the space of compact subsets of 
$\widehat{\mathbb{C}}$ equipped with Hausdorff $\rho$-metric is compact,  there is a convergent subsubsequence $\{K_{n_{k_m}}\}_m$, and a compact set $K$, such that $K_{n_{k_m}}\rightarrow K$. Considering (\ref{contradictioninequality}), limit set $K$ may satisfy

\begin{itemize}
\item[i.]  $d_E(u,K)=0$, i.e., $u\in K$, or
\item[ii.] $d_E(u,K)=\infty$, i.e., $K=\emptyset$.
\end{itemize}
In either of the cases above $\{(U_{n_{k_m}},u_{n_{k_m}})\}_m$ is divergent, by Theorem \ref{carconalter}. So, in order to rule out the possibilities given by $i.$ and $ii.$ the inequality (\ref{precompactnessbound1}) is a necessary condition for precompactness in $\mathcal{D}^*$.\\
\\
Now suppose (\ref{precompactnessbound1}) holds, and let $\{(U_{n_k},u_{n_k})\}_k$ be a subsequence in $\mathcal{D}^*$. We want to prove that there exists a subsubsequence $\{(U_{n_{k_m}},u_{n_{k_m}})\}_m$ which converges in the sense of Carath\'{e}odory. We use Theorem \ref{carconalter}. For $K_{n_k}=\widehat{\mathbb{C}}\backslash U_{n_k}$, after passing to a subsequence, $\{K_{n_{k_m}}\}_m$ converges to a compact set $K\in\widehat{\mathbb{C}}$, and $u_{n_{k_m}}\rightarrow u$. The inequality (\ref{precompactnessbound1}) implies $\frac{1}{L}\leq d_E(u,K)\leq L$. This assures $\widehat{\mathbb{C}}\backslash K$ has a connected component $U$, which is a proper subdomain containing $u$. Thus,  $(U_{n_{k_m}},u_{n_{k_m}})\stackrel{Car} \longrightarrow(U,u)$, by Theorem \ref{carconalter}.
\end{proof}
 
\subsection{Hyperbolic geometry and domain convergence}
\subsubsection{Preliminaries on hyperbolic geometry}
We restrict ourselves to the domains in the complex sphere $\widehat{\mathbb{C}}$. We recall that any domain in $\widehat{\mathbb{C}}$ which has at least three boundary points is called hyperbolic. Here we denote  by $d_{\Omega}(.,.)$, the hyperbolic distance induced by the hyperbolic metric defined in a hyperbolic set $\Omega$. \\
\\
The following two famous theorems have a key role in the proof of the Main Theorem. We start by the  Schwarz-Pick's Lemma, which tells that holomorphic maps have the property that they do not increase the hyperbolic metric. 
\begin{thm}\label{schpick}(Schwarz-Pick Lemma)
Let $\Omega_1, \Omega_2\subset\widehat{\mathbb{C}}$ be hyperbolic, and suppose $f:\Omega_1\rightarrow \Omega_2$ is a holomorphic map. Then for all $z,w\in \Omega_1$
\begin{equation}\label{SP}
d_{\Omega_2}(f(z),f(w))\leq d_{\Omega_1}(z,w).
\end{equation}

\noindent More precisely, if $f$ is a covering map, then it is a local hyperbolic isometry with respect to the two metrics. In this case, the hyperbolic arc lengths are preserved for the respective hyperbolic metrics. If $f$ is not a cover, then it 
 decreases the hyperbolic arc lengths and distances, that is, the strict inequality holds in (\ref{SP}).
\end{thm} 

\noindent For the proof, see, for example \cite{bearmin2000}. \\
\\
For a given pair of hyperbolic sets $\Omega_1, \Omega_2\subset\widehat{\mathbb{C}}$ with $\Omega_1\subset\Omega_2$, assigning the identity map $Id:\Omega_1\rightarrow\Omega_2$, the Schwarz-Pick Lemma gives the Comparison Principle, which we state in the following.
\begin{thm}\label{compprin} (Comparison Principle)
Let $\Omega_1, \Omega_2\subset\widehat{\mathbb{C}}$ be hyperbolic, and suppose that  $\Omega_1\subset\Omega_2$. Then for all $z,w\in\Omega_1$,
\begin{equation}
d_{\Omega_2}(z,w)\leq d_{\Omega_1}(z,w).
\end{equation} 
\end{thm}

\begin{lem}\label{lem3}
Let $\Omega\subset \widehat{\mathbb{C}}$ be hyperbolic, and let $\{x_n\}_n$ and $\{y_n\}_n$ be sequences in $\Omega$. Suppose  $d_{\Omega}(x_n,y_n)\leq \kappa<\infty$. If $x_n\rightarrow x\in\partial \Omega$, then $y_n\rightarrow x$. 
\end{lem}

\noindent For the general proof for Riemann surfaces, see \cite[Thm 3.4]{mil2006}. 
\subsubsection{Convergence of marked points and hyperbolic metric}
The sequence of marked points $\{u_n\}_n$ plays an important role in the convergence of domains. For example, a sequence of domains can be convergent or divergent depending on $\{u_n\}_n$. Moreover, in case of convergence, $\{u_n\}_n$ identifies the limit domain. In the following two lemmas, we analyze the convergence of domains in terms of the marked points,  by using hyperbolic geometry. We work in the space $\mathcal{D}^*$, which is given by (\ref{D^*}). 
\begin{lem}\label{thmsamecomponent} Let $U$ and $U_n$ be hyperbolic disks  in $\mathcal{D}^*$, and suppose $(U_n,u_n)\stackrel{Car} \longrightarrow (U,u)$. Suppose also that there is a sequence $\{w_n\}_n\rightarrow w$ such that $w_n\in U_n$ and $d_{U_n}(u_n,w_n)\leq \delta<\infty$,  for large $n$. Then $(U_n,w_n)\stackrel{Car} \longrightarrow(U,w)$.
\end{lem}
\begin{proof} 
 Let $\psi_n:\mathbb{D}\rightarrow U_n$ be univalent maps which satisfy $\psi_n(0)=u_n$, and $\psi_n'(0)>0$. Then the sequence $\{\psi_n\}_n$ converges locally uniformly to $\psi:\mathbb{D}\rightarrow U$ with $\psi(0)=u$ and $\psi'(0)>0$, by Theorem \ref{carthm}. Let $\widehat{w}_n\in\mathbb{D}$ be such that $\psi_n(\widehat{w}_n)=w_n$. We claim that for $\{\widehat{w}_n\}_n\rightarrow \widehat{w}$, $\psi(\widehat{w})=w$. Indeed, since $\psi_n$ are univalent, it preserves the hyperbolic distances, i.e., $d_{\mathbb{D}}(0,\widehat{w}_n)=d_{U_n}(u_n,w_n)\leq \delta$. Thus $\{\widehat{w}_n\}_n\subset\overline{\mathbb{D}(0,r)}$, where $\mathbb{D}(0,r)$ is the disk with Euclidian radius $r=\frac{e^{\delta}-1}{e^{\delta}+1}$, centered at $0$. Since $\overline{\mathbb{D}(0,r)}$ is compact, then there exists a convergent subsequence $\{\widehat{w}_{n_k}\}_k$, and say $\widehat{w}_{n_k}\rightarrow \widehat{w}$. On the other hand, $\psi_{n_k}\rightrightarrows\psi$ on compact sets of $\mathbb{D}$, and hence $\psi_{n_k}(\widehat{w}_{n_k})\rightarrow\psi(\widehat{w})=w$. Since $\psi$ is univalent, the limit $\widehat{w}$ is unique and is independent of the choice of the subsequence $\{\widehat{w}_{n_k}\}_k$, which means $\widehat{w}_n\rightarrow\widehat{w}$.\\
 \\
 Define the M\"{o}bius transformations $ M_n:\mathbb{D}\rightarrow\mathbb{D}$ as
\begin{equation*}
M(z):= e^{2\pi i\theta_n}\frac{z+\widehat{w}_n}{1+\overline{\widehat{w}_n}z},\;\;\;\;\theta_n=-Arg(\psi_n^{-1}(\widehat{w}_n))
\end{equation*} 
 Note that $(\psi_n\circ M_n)$ satisfies $(\psi_n\circ M_n)(0)=w_n$ and that $(\psi_n\circ M_n)'(0)>0$.  Set $M(z):= \frac{z+\widehat{w}}{1+\overline{\widehat{w}}z}$. Observe that as $\widehat{w}_n\rightarrow w$, $M_n\rightrightarrows M$ on compact subsets of $\mathbb{D}$. So the sequence $(\psi_n\circ M_n)_n$ converges uniformly to $(\psi\circ M)$ on compact sets of $\mathbb{D}$, with $(\psi\circ M)'(0)>0$. This implies Carath\'eodory convergence of the sequence of  pointed disks $\{(U_n,w_n)\}_n$ to $(U,w)$, by Theorem \ref{carthm}.
 \end{proof}

\begin{lem}\label{empty intersection} Let $U_n$, $Z$ and $W$ be hyperbolic sets in $\mathcal{D}^*$, and let $(U_n,w_n)\stackrel{Car} \longrightarrow(W,w)$ and  $(U_n,z_n)\stackrel{Car} \longrightarrow(Z,z)$. Suppose  $d_{U_n}(w_n,z_n)\rightarrow \infty$ as $n\rightarrow\infty$. Then $W\cap Z=\emptyset$.
\end{lem}
\begin{proof} Suppose $W\cap Z\neq \emptyset$. Take $v\in W\cap Z$.  Let $N_1$ be an open connected set, compactly contained in $W$, with $\{v,w\}\subset N_1$. Similarly let $N_2$ be an open connected set compactly contained in $Z$, with $\{z,v\}\subset N_2$. Then for sufficiently large $n$, $N_1\cup N_2\subset U_n$. The convergence in the sense of Carath\'eodory  implies $N_1\cup N_2\subset W\cap Z$, so $\{w,z\}\subset W\cap Z$. This contradicts that $d_{U_n}(w_n,z_n)\rightarrow\infty$ as $n\rightarrow\infty$.
\end{proof}

\section{Proof of the Main Theorem}
With the following proposition which is given in a very general setting, the Carath\'eodory topology meets with the dynamics.
\begin{prop}\label{lem4}
Let $f_n:\mathbb{C}\rightarrow\mathbb{C}$ be a sequence of nonlinear holomorphic functions converging locally uniformly to a nonlinear holomorphic function $f:\mathbb{C}\rightarrow\mathbb{C}$. Suppose $\{(U_n, u_n)\}_n$ is a sequence of simply connected pointed Fatou components of $f_n$ such that $(U_n,u_n)\stackrel{Car} \longrightarrow(U,u)$. Then $U$ is a simply connected subdomain of a Fatou component for $f$.
\end{prop}
\noindent This is a consequence of the fact that the Julia set depends lowersemicontinuously on the parameter (see \cite{dou1994}). 

\begin{proof} 
We prove by contradiction. Suppose $U$ is not contained in a Fatou domain for $f$. In this case $U\cap \mathcal{J}(f)\neq \emptyset$. So $U$ contains a repelling periodic point $w$ of some period $k$.  We are going to show that there is a repelling $k$-periodic point, say $w_n$ for $f_n$ near $w$. Choose $r>0$ small enough that in $\mathbb{D}(w,r)$, $w$ is the only $k$-periodic point in $\mathbb{D}(w,r)$. Set $\epsilon:=\inf_{z\in\overline{\mathbb{D}(w,r)}}|f(z)-z|$.
Since $f_{n}$ converges uniformly to $f$ on $\overline{\mathbb{D}(w,r)}$, there exists $N\in\mathbb{N}$ such that for all $z\in\overline{\mathbb{D}(w,r)}$ and for all $n\geq N$, $|f^k_{n}(z)-f^k(z)|\leq \epsilon$. Then by Rouch\'e's Theorem, there is exactly one $k$-periodic point, say $w_n$, for $f_{n}$ in $\mathbb{D}(w,r)$. Since the multiplier map is continuous,  $|(f^k)'(w)|>1$ implies  $|(f_n^k)'(w_n))|>1$, for $n$ sufficiently large, i.e., $w_n$ is a repelling $k$-periodic point for $f_n$. \\
\\
Let $N\subset U$ be an open, connected, relatively compact subset of $U$ containing $w$. For sufficiently large $n$, $N\subset U_n$. Since $w_n\rightarrow w$, increasing $n$ if necessary, $\{w_n,w\}\subset N\subset U_n$, which is a contradiction to the fact that $U_n$ is a Fatou domain. Thus, $U$ is contained in a Fatou component for $f$. Moreover, by Remark \ref{limitdomainsimplyconnected}, $U$ is simply connected. 
\end{proof}
\noindent In some cases, we may obtain the whole Fatou component in the limit, as shown in the following lemma. We now turn our attention to the family $f_a$ in consideration. Recall that $\mathcal{C}^0$ denotes the main hyperbolic component, and $A_a^0$ denotes the immediate basin of the superattracting fixed point $z=0$ for $f_a$.
\begin{lem}\label{abovelemma}
Let $\{a_n\}_{n}\subset {\mathcal{C}^0}$ be such that $a_n\rightarrow a\in\partial {\mathcal{C}^0}$. Then
 $(A_{a_n}^0,0)\stackrel{Car} \longrightarrow(A_a^0,0)$.
\end{lem}
\begin{proof}
Obviously for $a_n\in\mathcal{C}^0$, there exists $1<L<\infty$ such that $1/L\leq d_E(0,\partial A_{a_n}^0)\leq L$. Hence $\{(A_{a_n}^0,0)\}_n$ is precompact by Proposition \ref{precompactness1}. After passing to a subsequence if necessary, let $W_a$ be such that  $(A_{a_n}^0,0)\stackrel{Car} \longrightarrow(W_a,0)$. Then by Proposition \ref{lem4}, $W_a$ is a subdomain of a Fatou component, and moreover since $0\in W_a$, then $W_a\subset A_a^0$.
Let $\Omega_{a_n}$ be the maximal domain in $A_{a_n}^0$, where the restriction of the B\"{o}ttcher coordinate $\phi_{a_n}|_{\Omega_{a_n}}:\Omega_{a_n}\rightarrow\mathbb{D}_{r(a_n)}$ is biholomorphic. Define $\displaystyle{\psi_{a_n}:=\frac{z}{r(a_n)}}$. Then $(\psi_{a_n}\circ\phi_{a_n})(\Omega_{a_n})=\mathbb{D}$. We take $(\psi_{a_n}\circ\phi_{a_n})^{-1}:\mathbb{D}\rightarrow \Omega_{a_n}$ as the sequence of univalent maps associated to the sequence of pointed disks $\{(\Omega_{a_n},0)\}_{n}$. Since the Green's function is continuous with respect to the parameter,  as $a_n\rightarrow a\in\partial {\mathcal{C}^0}$,  $g_{a_n}(\infty)\rightarrow 0$, and hence $r(a_n)\rightarrow 1$. This means $ (\psi_{a_n}\circ\phi_{a_n})^{-1}\rightrightarrows (\phi_a)^{-1}$ on compact sets of $\mathbb{D}$, where $\phi_a$ is the B\"{o}ttcher map for $a\in\partial {\mathcal{C}^0}$. This implies $(\Omega_{a_n},0)\stackrel{Car} \longrightarrow(U_a,0)$ for some subdomain $U_a$ of a Fatou domain as in Proposition \ref{lem4}. Moreover, since $(\phi_a)^{-1}$ is defined everywhere in $\mathbb{D}$, then $U_a=A_a^0$. On the other hand $U_a\subset W_a$. Therefore we obtain $W_a=A_a^0$. 
\end{proof}
\begin{rem}
Convergence in Lemma \ref{abovelemma} is independent of how the sequence $\{a_n\}_{n}$ converges to the boundary point $a\in\partial {\mathcal{C}}^0$.
\end{rem}

\noindent Recall that $d_{A_a^0}(.,.)$ denote the hyperbolic distance in $A_a^0$. The following proposition relates the potentials of two points on a dynamic ray and the hyperbolic distance $d_{A_a^0}(.,.)$ between them.
\begin{prop}\label{comparablepotentials}
For any $c>1$ and $x,y\in R_{A_a^0}(\theta)$, $\theta\in \mathbb{R}\slash\mathbb{Z}$, if
\begin{equation*}
cg_a(x)\leq g_a(y)\leq g_a(x)\leq g_a(a)
\end{equation*} 
holds, then $d_{A_a^0}(x,y)<\log(2c)$.
\end{prop}
\begin{proof}
Suppose $\Omega_a$ as in the proof of Lemma  \ref{abovelemma}, and let $\Omega_a^*:=\Omega_a\backslash\{0\}$. By assumption, $x,y\in\Omega_a^*$.  Set $\mathcal{H}:=\{z:\Re z<g_a(\infty)\}$. Then $\exp:\mathcal{H}\rightarrow \phi_a(\Omega_a^*)$ is a universal cover. Let $\widehat{x},\widehat{y}\in \mathcal{H}$ be such that $\exp(\widehat{x})=\phi_a(x)$ and $\exp(\widehat{y})=\phi_a(y)$ with $\Im\widehat{x}=\Im\widehat{y}=2\pi\theta$. Since $\exp$ is a local hyperbolic isometry, and $\phi_a$ is conformal, the hyperbolic distance is preserved, i.e., $d_{\mathcal{H}}(\widehat{x},\widehat{y})=d_{\Omega_a^*}(x,y)$ (see Theorem \ref{schpick}). Moreover, since $\Omega_a^*\subset A_a^0$,  $d_{A_a^0}(x,y)< d_{\Omega_a^*}(x,y)$, by the Comparison Principle (see Theorem \ref{compprin}). Then we have
\begin{eqnarray*}
d_{A_a^0}(x,y)<d_{\mathcal{H}}(\widehat{x},\widehat{y})&=&\log\frac{\Re\widehat{y}-g_a(\infty)}{\Re\widehat{x}-g_a(\infty)}\\
&\leq&\log\frac{\Re\widehat{y}}{\Re\widehat{x}-g_a(\infty)}\leq\log\frac{2\Re\widehat{y}}{\Re\widehat{{x}}}\leq \log(2 c),
\end{eqnarray*}
substituting $\min \Re\widehat{x}=g_a(a)$, and $g_a(\infty)=\frac{1}{2}g_a(a)$.
\end{proof}
\noindent Before giving a corollary of this proposition, we need the following definition.
\begin{defn}
Let $\theta\in\mathbb{Q}\slash\mathbb{Z}$ be such that $2^q\theta\equiv \theta\pmod{1}$ for minimal $q\in\mathbb{N}$. Given $x\in R_{A_a^0}(\theta)$, \textit{the fundamental segment} with endpoints $x$ and $f_a^q(x)$ is the ray segment of $R_{A_a^0}(\theta)$ with the same endpoints, and we denote it by $R_{A_a^0}[x,f_a^q(x)]$.
\end{defn}
\begin{cor}\label{lemmain}
For $a\in R_{\mathcal{C}^0}(\theta)$, $\theta\in\mathbb{Q}\slash\mathbb{Z}$, such that $2^{l+q}\theta\equiv 2^{l}\theta \pmod{1}$ for some minimal integers $l\geq 0$, $q\geq 1$, observe that $f_a^{l}(a),f_a^{l+q}(a)\in R_{A_{a}^0}(2^{l+q}\theta)$. 
Proposition \ref{comparablepotentials} implies that the hyperbolic length of the fundamental segment $R_{A_a^0}[f_a^l(a),f_a^{l+q}(a)]\subset R_{A_a^0}(2^{l+q}\theta)$ is bounded with respect to the hyperbolic metric in $A_a^0$. Indeed, it verifies the hypothesis of Proposition \ref{comparablepotentials}, assigning $x=f_a^{l}(a)$ and $y=f_a^{l+q}(a)$, with $c=2^q$.
\end{cor}

\noindent For $a\in R_{\mathcal{C}^0}(\theta)$, $\theta\in\mathbb{Q}\slash\mathbb{Z}$, such that $2^{l+q}\theta\equiv 2^{l}\theta \pmod{1}$, we divide the proof of the Main Theorem into three cases; when $l=0$, $l=1$ and $l>1$, given in the following three sections. By the (pre-)periodic case, we mean the case when the asymptotic value is on a (pre-)periodic dynamic ray.
\subsection{Periodic case-I (when $l=0$)}
\begin{lem}\label{inftyaccumulationpoint} Let $\theta$ be such that $2^q\theta\equiv\theta\pmod{1}$ for some minimal $q\in\mathbb{N}$. Let  $a_n\in R_{\mathcal{C}^0}(\theta)$, and suppose that $a_n\rightarrow a\in\partial \mathcal{C}^0$. For each $0\leq m\leq q-1$, the sequence of pointed domains $\{(A_{a_n}^0,f_{a_n}^{m}(a_n))\}_{n}$ is precompact in the sense of Carath\'eodory.
\end{lem}
\begin{proof} For simplicity, we will show only the case $m=0$.  We argue by contradiction. First observe that $a_n\nrightarrow\infty$, since $\mathcal{C}^0$ is bounded (see Theorem \ref{mainhyperboliccomponentisbounded}). It is also obvious that $d_E(a_n,\partial A_{a_n}^0)\nrightarrow \infty$, since the Julia set depends lowersemicontinuously on the parameter, and $\mathcal{J}(f_{a})\neq\emptyset$.
Suppose the sequence $\{(A_{a_n}^0,a_n)\}_{n}$ is not precompact. This implies $d_E(a_n,\partial A_{a_n}^0)\rightarrow 0$ as $a_n\rightarrow a\in\partial {\mathcal{C}^0}$ by Proposition \ref{precompactness1}, together with the discussion above. Moreover, by Corollary \ref{lemmain}, $d_{A_n^0}(a_n,f_{a_n}^q(a_n))<L$ for some $0<L<\infty$. Hence  $f_{a_n}^q(a_n)\rightarrow a$, by Lemma \ref{lem3}. On the other hand, $f_{a_n}(R_{A_{a_n}^0}[\infty, f_{a_n}^{q-1}(a_n)])=R_{A_{a_n}^0}[a_n,f_{a_n}^q(a_n)]$, thus $f_a(R_{A_{a}^0}[\infty,f_a^{q-1}(a)])=\{a\}$. This implies that $f_a$ is either constant by the Identity Theorem, or $f_a^{q-1}(a)=\infty$, which is a contradiction. Therefore, there exists $0<L<\infty$, such that $d_E(a_n,\partial A_{a_n}^0)>L$, which assures precompactness. 
\end{proof}
\noindent Lemma \ref{inftyaccumulationpoint} implies Carath\'eodory convergence of the pointed disks $\{(A_{a_{n}}^0,a_{n})\}_{n}$, after passing to a subsequence, and the limit domain $U_a^0$ is a subdomain of a Fatou domain by Proposition \ref{lem4}. 
\begin{lem}\label{limitdomain}  Let $\theta$ be such that $2^q\theta\equiv\theta\pmod{1}$ for some minimal $q\in\mathbb{N}$. Let  $a_n\in R_{\mathcal{C}^0}(\theta)$, and suppose that $a_n\rightarrow a\in\partial \mathcal{C}^0$. Suppose $(A^0_{a_n},a_n)\stackrel{Car} \longrightarrow (U_a^0, a)$. Then $U_a^0$ is a subdomain of the connected component of a  parabolic immediate basin, which is fixed by $f_a^q$.

\end{lem}
\begin{proof}
Since $U_a^0$ is a subdomain of a Fatou component by Proposition \ref{lem4}, and  since $a\notin A_a^0$ and $a\in U_a^0$, then $U_a^0\cap A_a^0=\emptyset$. Moreover, by Lemma \ref{thmsamecomponent},  $f_a^q(a)\in U_a^0$, since $d_{A_{a_n}^0}(a_n,f_{a_n}^q(a_n))$ is bounded (see Corollary \ref{lemmain}). This means $U_a^0$ is contained in a $q$-periodic Fatou component. Moreover, there are no more free singular values, and hyperbolicity is an open condition. Therefore for $a\in\partial {\mathcal{C}^0}$, the limit domain $U_a^0$ is in a $q$-periodic parabolic component. By Lemma \ref{inftyaccumulationpoint}, $\{(A_{a_n}^0,f^m_{a_n}(a_n))\}_{n}$ for $m=1,...,(q-1)$ is also precompact, which means there exists a limit domain $U_a^m$, passing to a subsequence. Therefore, the set $\{U_a^i\}$, $i=0,1,...,(q-1)$ forms a cycle of subdomains of the immediate parabolic basin.
\end{proof}

\noindent Note that Lemma \ref{limitdomain} does not imply immediately that $q$ is the exact period of the parabolic basin (compare Proposition \ref{exactperiod}).
\subsubsection{Hausdorff convergence of dynamic rays}
Let $\theta$ be such that $2^q\theta\equiv\theta\pmod{1}$ for some minimal $q\in\mathbb{N}$. Observe that
for $a_n\in R_{\mathcal{C}^0}(\theta)$,  $a_n\in R_{A_{a_n}^0}(2\theta)$. Recall that we denote the maximal domain by $\Omega_{a_n}\subset A_{a_n}^0$, for which the restriction of the  B\"{o}ttcher map $\phi_{a_n}|_{\Omega_{a_n}}:\Omega_{a_n}\rightarrow\mathbb{D}_{r(a_n)}$ is biholomorphic. Recall also that for such $\theta$ and $a_n\rightarrow a\in\partial \mathcal{C}^0$, $(A^0_{a_n},a_n)\stackrel{Car} \longrightarrow (U_a^0, a)$, where $U_a^0$ is a subdomain of the connected component of a  parabolic immediate basin, by Lemma \ref{limitdomain}. Let us denote the parabolic basin for $f_a$ by $B_a$. Denote the ray segment in $R_{A_{a_n}^0}(2\theta)$, which connects $0$ and $a_n$ by $K_{a_n}$, i.e., $K_{a_n}:=\phi_{a_n}^{-1}\{re^{2\pi i 2\theta},0\leq r\leq (r(a_n))^2\}\subset R_{A_{a_n}^0}(2\theta)$.  We investigate the Hausdorff limit of  $K_{a_n}$ as $a_n\rightarrow a\in \partial \mathcal{C}^0$.  
Since the space of compact subsets of $\widehat{\mathbb{C}}$ equipped with the Hausdorff $\rho$-metric is compact, passing to a subsequence if necessary, $\{K_{a_n}\}_n$ converges to a compact connected set $K_a\subset\widehat{\mathbb{C}}$ as $a_n\rightarrow a$. Here $K_a$ connects $0$ and $a$. 

\begin{prop}\label{realanalyticarc}
Any connected component $\gamma$ of $K_a\cap \mathcal{F}(f_a)$ is a real analytic arc, and either:
\begin{itemize}
\item[i.] $\gamma=R_{A_a^0}(2\theta)$, 
\item[ii.] $\gamma\subset U_a^0$, and there exists a univalent map $\Xi_{\infty}:\widetilde{\mathcal{H}}\rightarrow \mathbb{C}$ from a right half plane $\widetilde{\mathcal{H}}=\{x+iy: x>t_{\infty}, y\in\mathbb{R}\}$, such that $\gamma=\Xi_{\infty}[t_{\infty},\infty)$, for some $0<t_{\infty}<1/2$ with $\Xi_{\infty}(1)=a$, and $\Xi_{\infty}$ conjugates $z\mapsto 2^q z$ to $f_a^q$, or
\item[iii.] $\gamma\subset B_a\backslash U_a^0$, and there exists a univalent map $\Xi_{\infty}:\mathbb{H}_{+}\rightarrow \mathbb{C}$, $\mathbb{H}_{+}=\{z;\;\;\Re z>0\}$, such that $\Xi_{\infty}$ conjugates $z\mapsto 2^q z$ to $f_a^q$, and $\gamma=\Xi_{\infty}(\mathbb{R}_{+})$.

\end{itemize}
\end{prop}
\begin{proof}
Set $\widetilde{\Omega}_{a_n}:=\Omega_{a_n}\backslash R_{A_{a_n}^0}(\frac{1}{2}+2\theta)$. Then the logarithm is well defined in $\widetilde{\mathbb{D}}_{r(a_n)}:=\phi_{a_n}(\widetilde{\Omega}_{a_n})$. Let $z(a)\in K_a$, and $z_n\in K_{a_n}$ be such that $z_n\rightarrow z(a)$, as $a_n\rightarrow a\in\partial \mathcal{C}^0$. We define univalent maps, using the  B\"{o}ttcher coordinates:
\begin{eqnarray*}
\Psi_n:\widetilde{\Omega}_{a_n}&\rightarrow & \mathbb{H}_{+},\notag\\
z&\mapsto& \frac{\log\phi_{a_n}(z)-\Im\log\phi_{a_n}(z_n)}{\Re\log\phi_{a_n}(z_n)}.
\end{eqnarray*}
 Observe the following two situations:
\begin{itemize}
\item[i.] If $\Re\log(\phi_{a_n}(z_n))\nrightarrow 0$, then the range $\widehat{\Omega}:=\Psi_n(\widetilde{\Omega}_{a_n})$ converges to a horizontal strip in some right half plane,
\item[ii.] If $\Re\log(\phi_{a_n}(z_n))\rightarrow 0$, then the range $\widehat{\Omega}$ converges to a right half plane.
\end{itemize}
In either case, the map $\Psi_n$ sends the ray $R_{A_{a_n}^0}(2\theta)\cap\widetilde{ \Omega}_{a_n}$ to the half line $(t_n,\infty)$, where $t_n:=\frac{1}{2}\frac{g_{a_n}(a_n)}{g_{a_n}(z_n)}$, and in particular, $\Psi_n(z_n)=1$. The map $\Psi_n$ conjugates $f_{a_n}^q$  to $z\mapsto 2^q z$. Let $\Xi_n:=\Psi_n^{-1}$ denote the univalent map defined in the range of $\Psi_n$ into $A_{a_n}^0$. Obviously $\{\Xi_n\}_n$ is a normal family, hence passing to a subsequence if necessary, it is uniformly convergent to a map $\Xi_{\infty}$, on compact sets of $\widehat{\Omega}$. The limit map $\Xi_{\infty}$ is either univalent or  constant. We claim that $\Xi_{\infty}$ is univalent. Indeed, if it is constant, say $\Xi_{\infty}\equiv c$, then by conjugacy, $c$ is a $q$ periodic point in a Fatou domain. This implies $c$ is in an attracting cycle, which contradicts $a\in\partial \mathcal{C}^0$. Hence $\Xi_{\infty}$ is univalent, and this assures convergence of $\{(\widetilde{\Omega}_{a_n},z_n)\}_n$ in the sense of Carath\'{e}odory as $a_n\rightarrow a\in\partial \mathcal{C}^0$, after passing to a subsequence. In this case,  $\Xi_{\infty}$ conjugates $z\mapsto 2^q z$ to $f_{a}^q$. This means, $\gamma$ is a real analytic arc, along which the direction of the dynamics $z\mapsto f_a^q(z)$ is preserved. \\
\\
Assume  $z_n\rightarrow z(a)\in K_a\cap \mathcal{F}(f_a)$. By hypothesis, $g_{a_n}(z_n)<g_{a_n}(a_n)$.
\begin{itemize}
\item[i.] If $g_{a_n}(z_n)<c<0$, then $g_a(z(a))<c$, and  $\Psi_n$ converges to a conformal map onto
\begin{equation*} 
H:=\{x+yi, x>0,-\frac{\pi}{g_a(z(a))} <y<\frac{\pi}{g_a(z(a))}\},
\end{equation*}
which is defined in whole $A_a^0\backslash R_{A_a^0}(\frac{1}{2}+2\theta)$.  By construction, $z(a)\in A_a^0$, $\Xi_{\infty}(0,\infty)=R_{A_a^0}(2\theta)$.
\item[ii.] If $1<\frac{g_{a_n}(a_n)}{g_{a_n}(z_n)}<c$, for some $1<c<\infty$, then $z(a)$ is in $U_a^0$, by Lemma \ref{thmsamecomponent} and Proposition \ref{comparablepotentials}, and  $\gamma=\Xi_{\infty}(t_{\infty},\infty)$ is a real analytic arc in $K_a\cap U_a$. Observe that $\Xi_{\infty}$ is defined in a right half plane 
\begin{equation*}
\widetilde{H}:=\{x+iy;\;\;x>t_{\infty},y\in\mathbb{R}\}.
\end{equation*}
By a suitable Mobius transformation, we can also have $\Xi_{\infty}(1)=a$, as in the statement of the Proposition. In this case, $0<t_{\infty}<1/2$. 
\item[iii.] If $\frac{g_{a_n}(a_n)}{g_{a_n}(z_n)}\rightarrow 0$, and $g_{a_n}(z_n)\rightarrow 0$, then $d_{A_{a_n}^0}(a_n,z_n)\rightarrow\infty$. This means $z(a)$ and $a$ are not contained in the same limiting domain by Lemma \ref{empty intersection}. Moreover $z(a)$ will be in a periodic component by Lemma \ref{thmsamecomponent}. Hence $z(a)$ is  in $B_a\backslash U_a^0$ because there is no other Fatou component apart from $A_a$, or $B_a$. Since $\frac{g_{a_n}(a_n)}{g_{a_n}(z_n)}\rightarrow 0$, we have $t_{\infty}=0$, and $\gamma=\Xi_{\infty}(\mathbb{R}_+)$.
\end{itemize}
\end{proof}

\subsubsection{Limit domains and landing points}
Let $W$ be a Fatou component such that $K_a\cap W\neq\emptyset$. Then $W$ is either in the superattracting basin $A_a$, or in the parabolic basin $B_a$. Indeed, if not, it requires at least one more free singular value. But $f_a$ has only two singular values, one of each is included in either $A_a$, or $B_a$. Therefore, there is no other basin to intersect with $K_a$.

\begin{prop}\label{exactperiod}
The limiting domain $U_a^0$ is contained in a $q$-periodic immediate parabolic basin, where the period $q$ is exact.
\end{prop}
\begin{proof}
Observe the following:
\begin{itemize}
\item[i.] No two rays in $A_a^0$ land at the same finite point. We will show it by contradiction. Suppose two dynamic rays, $R_{A_a^0}(\alpha_1)$, and $R_{A_a^0}(\alpha_2)$, $\alpha_1,\alpha_2\in\mathbb{R}\slash\mathbb{Z}$ coland at a finite point $z_0\in\partial A_a^0$. In that case, the immediate basin $A_a^0$ is illustrated by Figure \ref{maxmod1}. Set $\Gamma:=\overline{R_{A_a^0}(\alpha_1)\cup R_{A_a^0}(\alpha_2)}$. Note that $\Gamma\cap \mathcal{J}(f_a)=\{z_0\}$ and $\Gamma\backslash\{z_0\}\subset A_a^0$. A finite number of  forward iterates of $z_0$ stay bounded. Thus forward iterates of $\Gamma$ will be bounded. However the domain bounded by $\Gamma$ contains a subset of the Julia set. This contradicts the Maximum Modulus Principle. This observation particularly implies that no two periodic rays in $A_a^0$ land at the same point, since the landing point is periodic. 

\begin{figure}[htb!]
\begin{center}
\def\svgwidth{6 cm}
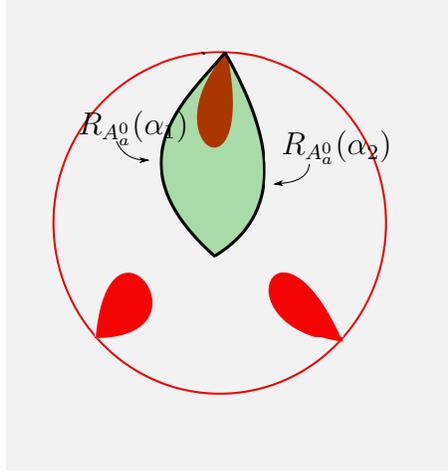
\caption{ \small{In the model figure of $A_a^0$, the shaded region bounded by $\Gamma=\overline{R_{A^0}(\alpha_1)\cup R_{A_a^0}(\alpha)}$ contains some Julia set, which contradicts the Maximum Modulus Principle.}}\label{maxmod1}
\end{center}
\end{figure}

\item[ii.]
Each limiting domain $\{U_a^i\}$, $0\leq i<q$  is in a distinct component of the immediate parabolic basin. Suppose two domains $U_a^m$, and $U_a^n$ are in the same component, say $B_a^n$. The situation is illustrated in Figure \ref{maxmod2}. Take a curve  $\widehat{\gamma} \subset B_a^n$, connecting $f_a^m(K_a)$ and $f_a^{n}(K_a)$, for $0<m,n<q$. Then the curve $\Gamma:=\widehat{\gamma}\cup f_a^{m}( K_a)\cup f_a^{n}(K_a)$ contains only two points in the Julia set, which are $q$-periodic. Hence the $qth$ iterate of the curve stays bounded. However the region bounded by $\Gamma$ contains a subset of the Julia set, which  contradicts  the Maximum Modulus Principle. 
 
\begin{figure}[htb!]
\begin{center}
\def\svgwidth{6 cm}
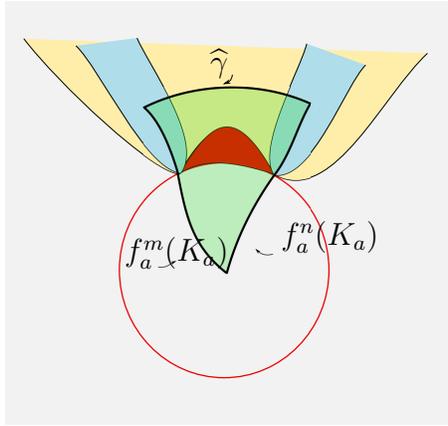
\caption{\small{In the model figure, the shaded region bounded by $\Gamma=\widehat{\gamma}\cup f_a^{m}(K_a)\cup f_a^{n}(K_a)$ contains some Julia set, which contradicts the Maximum Modulus Principle.}}\label{maxmod2}
\end{center}
\end{figure}

 \end{itemize}

\noindent It then follows from $i.$ and $ii.$ that the exact period of the immediate parabolic basin is $q$.
\end{proof}



\noindent Observe that for $a\in\partial \mathcal{C}^0$, $K_a\cap \mathcal{J}(f_a)\neq\emptyset$, since $K_a$ is a connected set which connects $0$ and $a$. Now, we will specify the character of the periodic points at $\mathcal{J}(f_a)\cap K_a$.

\begin{prop}\label{singlepointonboundary}
Every connected component of $K_a\cap \mathcal{J}(f_a)$ is a $q$-periodic point.
\end{prop}
\begin{proof}
 For any point $z(a)\in K_a\cap \mathcal{J}(f_a)$, there exist some $z_n\in K_{a_n}\subset A_{a_n}^0$, such that $z_n\rightarrow z(a)$. Those points also satisfy $d_{A_{a_n}^0}(z_n,f_{a_n}^q(z_n))<q\log 2$, given by Corollary \ref{lemmain}. Thus, if $z_n\rightarrow z(a)$, then $f_{a_n}^{q}(z_n)\rightarrow z(a)$, which means 
 \begin{equation}\label{discreteset}
 f_{a}^q(z(a))=z(a)
 \end{equation}
by Lemma \ref{lem3}.
 \end{proof}
\begin{cor}\label{kissing}
Let $\{U_a^i\}$, $0\leq i< q$ be such that $(A^0_{a_n},f_{a_n}^i(a_n))\stackrel{Car} \longrightarrow (U_a^i, f_a^i(a))$ (see Lemma \ref{inftyaccumulationpoint}). Each $U_a^i$ connects with $A_a^0$ at a periodic point on their common boundaries.
\end{cor}

\begin{prop}\label{multiplierone}
Every point $z(a)\in K_a\cap \mathcal{J}(f_a)$ is a parabolic periodic point with $f_a'(z(a))=1$.
\end{prop}

\begin{proof}
We prove it by contradiction. Suppose $z(a)\in K_a\cap \mathcal{J}(f_a)$ is not parabolic point, i.e., $z(a)=\mathcal{J}(f_a)\cap K_a$ is a repelling $k$-periodic point. For a sequence $\{a_n\}_n\subset\mathcal{C}^0$, such that $a_n\rightarrow a\in\partial \mathcal{C}^0$, there exists $r>0$, $N\in\mathbb{N}$ such that there exists a unique repelling $k$-periodic point for $f_{a_n}$ in  $\mathbb{D}(z(a),r)$. This follows from Rouch\'e's Theorem (see the proof of Proposition \ref{lem4} for details). \\
\\
On the other hand, since $z(a)\in K_a$, there exists $\{w_n\}_n$, $w_n\in K_{a_n}$, such that $w_n\rightarrow z(a)$ as $n\rightarrow\infty$. Now we consider the two sequences $\{w_n\}_n$ and $\{z(a_n)\}_n$, both converging to $z(a)$ as $n\rightarrow\infty$. Note that $w_n\in A_{a_n}^0$, whereas $z_n\notin A_{a_n}^0$. This means $w_n\rightarrow \partial A_{a_n}^0$ as $n\rightarrow\infty$, since $|w_n-z(a_n)|\rightarrow 0$. \\
\\
Let $m\in\mathbb{N}$ be such that $w_n$ is contained in $R_{A_{a_n}^0}[f_{a_n}^{km}(a_n),f_{a_n}^{k(m+1)}(a_n)]$.  Here $m$ is a finite number, which depends on finite $n$. Observe that 

\begin{equation}
d_E(f_{a_n}^{km}(a_n),f_{a_n}^{k(m+1)}(a_n))\rightarrow 0
\end{equation}
 as $n\rightarrow \infty$, since fundamental segments have bounded hyperbolic length in the hyperbolic metric defined in $A_{a_n}^0$ (compare Proposition \ref{comparablepotentials} and Corollary \ref{lemmain}), together with $w_n\rightarrow\partial A_{a_n}^0$ (see Lemma \ref{lem3}).  Note also that any fixed finite number of adjacent fundamental segments on $R_{A_{a_n}^0}(2\theta)$ near $R_{A_{a_n}^0}[f_{a_n}^{km}(a_n),f_{a_n}^{k(m+1)}(a_n)]$ also shrink as $n\rightarrow\infty$ (see again Lemma \ref{lem3}).\\
\\
For $w_n\in\mathbb{D}(z(a),r)$,  increasing $n$ if necessary, we can assume that 
\begin{equation}
R_{A_{a_n}^0}[f_{a_n}^{km}(a_n),f_{a_n}^{k(m+1)}(a_n)]\subset f_{a_n}^k(\mathbb{D}(z(a),r)),
\end{equation}
since  the Euclidian length of $R_{A_{a_n}^0}[f_{a_n}^{km}(a_n),f_{a_n}^{k(m+1)}(a_n)]$ is small for large $n$. \\
\\
The unique lift of $R_{A_{a_n}^0}[f_{a_n}^{km}(a_n),f_{a_n}^{k(m+1)}(a_n)]$ by $f^{k}_{a_n}$ on the same ray, i.e., the preimage with endpoint $f_{a_n}^{km}(a_n)$  is contained in $\mathbb{D}(z(a),r)$. By induction, we see that all iterated preimages of $R_{A_{a_n}^0}[f_{a_n}^{km}(a_n),f_{a_n}^{k(m+1)}(a_n)]$ by $f_{a_n}^k$ on the same ray are contained in $\mathbb{D}(z(a),r)$, so is the ray segment $R_{A_{a_n}^0}[f_{a_n}^{mk}(a_n) ,a_n]$.\\
\\
Observe that the Hausdorff limit of $R_{A_{a_n}^0}[f_{a_n}^{mk}(a_n),a_n]$, say $\gamma$, is a compact set with distinct endpoints $z(a)$ and $a$. So $\gamma$ has a certain Euclidian length. Note that $n\rightarrow\infty$, $m\rightarrow\infty$. This is because as $f_{a_n}^{mk}(a_n)\rightarrow \partial A_{a_n}^0$, the Euclidian length of the fundamental segments along the ray $R_{A_{a_n}^0}(2\theta)$ near the boundary of $A_{a_n}^0$ will get close to $0$. This means that the number of iterates required in order to reach from $a_n$ to $f_{a_n}^{mk}(a_n)$ tends to $\infty$. Since $z(a)$ is the only (repelling) $k$-periodic point in $f_{a_n}^k(\mathbb{D}(z(a),r))$, then $a_n\rightarrow z(a)$ as $n\rightarrow\infty$. This means $a=z(a)$, which is a contradiction to the fact that $a\in U_a^0$. So  $z(a)\in K_a\cap \mathcal{J}(f_a)$ is a parabolic periodic point, and by the Snail Lemma $(f_a^k)'(z(a))=1$. 


\end{proof}

\subsection{Periodic case-II (when $l=1$)}\label{preperiodiccase}
\begin{lem}\label{othercase} Let $\theta$ be such that $2^{1+q}\theta\equiv2\theta\pmod{1}$ for some minimal $q\in\mathbb{N}$, and let  $a_n\in R_{\mathcal{C}^0}(\theta)$. Suppose that $a_n\rightarrow a\in\partial \mathcal{C}^0$. Then $\{(A_{a_n}^0,a_n)\}_n$ is not precompact, in the sense of Carath\'eodory.
\end{lem}

\begin{proof} We prove by contradiction. Suppose $\{(A_{a_n}^0, f_{a_n}^{l}(a_n))\}_n$ is precompact in the sense of Carath\'eodory. Take an arbitrary subsequence  $\{(A_{a_{n_k}}^0, f_{a_{n_k}}^{l}(a_{n_k}))\}_k$, and suppose $(A_{a_{n_{k_m}}}^0,f^{l}_{a_{n_{k_m}}}(a_{n_{k_m}}))\stackrel{Car} \longrightarrow (U_a^0,f_a^{l}(a))$. Here $U_a^0$ is periodic with period $q$  (see Lemma \ref{limitdomain}, Proposition \ref{exactperiod}). This cycle of components forms a cycle of subdomains of an immediate basin, disjoint from $A_a^0$. Also  the essential singularity and $z=1$ are included in the same component of the immediate basin. Let $\widehat{K}_a$ and $\widetilde{K}_a$ denote the two preimages of $K_a$ under $f_a$. Take a curve $\widetilde{\gamma}\subset U_a^0$ which connects $\widehat{K}_a$ and $\widetilde{K}_a$ at some finite points. Then the region bounded by $\widehat{\gamma}$, $\widehat{K}_a$, and $\widetilde{K}_a$ contains a subset of the Julia set, whereas, its boundary contains only two points in the Julia set, which are periodic (see Figure \ref{maxmod3} for $m=0$). This contradicts the Maximum Modulus Principle. Hence we conclude that $\{(A_{a_n}^0,a_n)\}_n$ is not precompact in the sense of Carath\'eodory.  Thus $\{(A_{a_n}^0,f_{a_n}^{l}(a_n))\}_n$ is not precompact in the sense of Carath\'eodory.

\end{proof}

\begin{figure}[htb!]
\begin{center}
\def\svgwidth{6 cm}
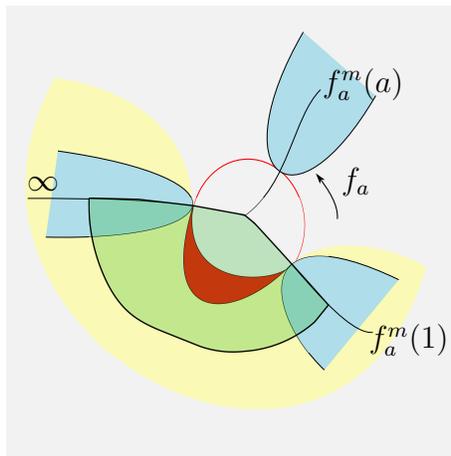
\caption{\small{The shaded  region in the figure contains some Julia set.} }\label{maxmod3}
\end{center}
\end{figure} 

\subsection{Preperiodic case (when $l>1$)}
\begin{lem}\label{othercase2}
Let $\theta$ be such that $2^{l+q}\theta\equiv2^{l}\theta\pmod{1}$ for some minimal integers $l>1$ and $q>1$, and let  $a_n\in R_{\mathcal{C}^0}(\theta)$. Suppose that $a_n\rightarrow a\in\partial \mathcal{C}^0$.  The sequence $\{(A_{a_n}^0, f_{a_n}^{l}(a_n))\}_n$ is not precompact in the sense of Carath\'eodory.
\end{lem}
\begin{proof}
We use the same approach as in Lemma \ref{othercase}. Suppose $\{(A_{a_n}^0, f_{a_n}^{l}(a_n))\}_n$ is precompact in the sense of Carath\'eodory. For an arbitrary $n_k$, consider the convergent subsubsequence $(A_{a_{n_{k_m}}}^0,f^{l}_{a_{n_{k_m}}}(a_{n_{k_m}}))\stackrel{Car} \longrightarrow (U_a^0,f_a^{l}(a))$. Here $U_a^0$ is periodic with period $q$. This cycle of components forms a cycle of subdomains of an immediate basin, disjoint from $A_a^0$. Similar to Lemma \ref{othercase}, the existence of the essential singularity in the periodic cycle of components contradicts the Maximum Modulus Principle (see Figure \ref{maxmod3} for $m>0$).  So $\{(A_{a_n}^0,f_{a_n}^{l}(a_n))\}_n$ is not precompact in the sense of Carath\'eodory.
\end{proof}

\begin{cor}
As a consequence of Lemma \ref{othercase} and Lemma \ref{othercase2}, 
\begin{equation*}
d_E(f_{a_n}^l(a_n),\partial A_{a_n}^0)\rightarrow 0,\;\;\mathrm{ and}\;\;d_{A_{a_n}^0}(f_{a_n}^l(a_n),f_{a_n}^{l+q}(a_n))<L,
\end{equation*}
 for $l\geq 0$ (as given by Corollary \ref{lemmain}) implies $d_E(f_{a_n}^l(a_n),f_{a_n}^{l+q}(a_n))\rightarrow 0$, and thus $f_a^{l}(a)=f_a^{l+q}(a)$.
\end{cor}
\subsection{Conclusion of the proof of the Main Theorem}
Now we collect all the results and complete the proof of the Main Theorem:
\\
\\
Any accumulation point $a\in\partial\mathcal{C}^0$ of a parameter ray with a rational argument satisfies either
\begin{eqnarray*}
f_a^q(z(a))&=&z(a),\notag\\
(f_a^q)'(z(a))&=&1
\end{eqnarray*}
as given by Proposition \ref{multiplierone}, or
\begin{equation*}
f_a^{q+l}(a)=f_a^l(a)
\end{equation*}
as given by Lemma \ref{othercase} and Lemma \ref{othercase2}.\\ 
\\
In both cases, the solution set forms a discrete set. However, the accumulation set is connected as being an accumulation set of a connected set. Therefore, the accumulation set consists of only one point. This proves all internal parameter rays with rational argument in the main hyperbolic component ${\mathcal{C}^0}$ lands at $\partial\mathcal{C}^0$.



\bibliography{references}
\bibliographystyle{plain}
\end{document}